\documentclass[a4paper,12pt]{article}
\usepackage{amsthm,amsmath,amssymb,latexsym,bm}

\newtheorem{thm}{Theorem}[section]

\newtheorem{lem}[thm]{Lemma}

\newtheorem{rem}[thm]{Remark}
\numberwithin{equation}{section}
\title{Higher order Painlev\'{e} system of type $D^{(1)}_{2n+2}$ and monodromy preserving deformation}
\author{Kenta Fuji \thanks{Department of Mathematics, Kobe University, 1-1, Rokkodai, Nada-ku, Kobe 657-8501, Japan. E-mail: fuji@math.kobe-u.ac.jp} \and Keisuke Inoue \and Keisuke Shinomiya \and Takao Suzuki \thanks{Department of Mathematics, Kinki University, 3-4-1, Kowakae, Higashi-Osaka, Osaka 577-8502, Japan. E-mail: suzuki@math.kindai.ac.jp}}
\date{}

\begin{document}

\maketitle

% Abstract
\begin{abstract}
The higher order Painlev\'{e} system of type $D^{(1)}_{2n+2}$ was proposed by Y. Sasano as an extension of $P_{\rm{VI}}$ for the affine Weyl group symmetry with the aid of algebraic geometry for Okamoto initial value space.
In this article, we give it as the monodromy preserving deformation of a Fuchsian system.


\end{abstract}

% Section 1
\section{Introduction}

The main object in this article is the higher order Painlev\'{e} system of type $D^{(1)}_{2n+2}$ \cite{Sas}; we call it {\it a Sasano sysytem}.
It is expressed as a Hamiltonian system on $\mathbb{P}^1(\mathbb{C})$
\begin{equation}\begin{split}\label{Eq:CP6}
	&s(s-1)\frac{dq_i}{ds} = \frac{\partial H}{\partial p_i},\quad
	s(s-1)\frac{dp_i}{ds} = -\frac{\partial H}{\partial q_i}\quad (i=1,\ldots,n),\\
	&H = \sum_{i=1}^{n}H_i + \sum_{1\leq i<j\leq n}2(q_i-s)p_iq_j\{(q_j-1)p_j+\alpha_{2j}\},
\end{split}\end{equation}
where $H_i=H_i[q_i,p_i;\kappa_{i,s},\kappa_{i,1},\kappa_{i,0},\kappa_{i,\infty};s]$ is the Hamiltonian for $P_{\rm{VI}}$ defined by
\[\begin{split}
	H_i &= q_i(q_i-1)(q_i-s)p_i^2 - (\kappa_{i,s}-1)q_i(q_i-1)p_i\\
	&\quad - \kappa_{i,1}q_i(q_i-s)p_i - \kappa_{i,0}(q_i-1)(q_i-s)p_i + \alpha_{2i}(\alpha_{2i}+\kappa_{i,\infty})q,
\end{split}\]
and
\begin{equation}\begin{split}\label{Eq:CP6_Param}
	&\kappa_{i,s} = \alpha_1 + \sum_{j=1}^{i-1}\alpha_{2j+1},\quad
	\kappa_{i,1} = \sum_{j=i}^{n-1}\alpha_{2j+1} + \sum_{j=i+1}^{n}2\alpha_{2j} + \alpha_{2n+1},\\
	&\kappa_{i,0} = \sum_{j=i}^{n-1}\alpha_{2j+1} + \alpha_{2n+2},\quad
	\kappa_{i,\infty} = \alpha_0 + \sum_{j=1}^{i-1}2\alpha_{2j} + \sum_{j=1}^{i-1}\alpha_{2j+1}.
\end{split}\end{equation}
The fixed parameters $\alpha_0,\ldots,\alpha_{2n+2}$ satisfy a relation $\alpha_0+\alpha_1+\sum_{j=2}^{2n}2\alpha_j+\alpha_{2n+1}+\alpha_{2n+2}=1$.

The system \eqref{Eq:CP6} was proposed as an extension of $P_{\rm{VI}}$ for the affine Weyl group symmetry with the aid of algebraic geometry for Okamoto initial value space.
It was also given as the compatibility condition of the Lax pair associated with a loop algebra $\mathfrak{so}_{4n+4}[z,z^{-1}]$ \cite{FS1}.
But the relationship with the monodromy preserving deformation of a Fuchsian system has not been clarified.
The aim of this article is to investigate it.

Recently, higher order generalizations of $P_{\rm{VI}}$ has been studied from a viewpoint of the monodromy preserving deformations of Fuchsian systems.
It is shown in \cite{Kos,O} that any irreducible Fuchsian system can be reduced to finite types of systems by using Katz's two operations, addition and middle convolution \cite{Kat}.
It is also shown in \cite{HF} that the isomonodromy deformation equation is invariant under Katz's two operations.
Those fact allows us to construct a classification theory of the isomonodromy deformation equation.

The Fuchsian systems with two accessary parameters are classified by Kostov \cite{Kos}.
According to it, they are reduced to the systems with the following spectral types:
\[\begin{array}{l|lll}
	\text{4 singularities}& 11,11,11,11\\\hline
	\text{3 singularities}& 111,111,111& 22,1111,1111& 33,222,111111
\end{array}\]
The system with the spectral type $\{11,11,11,11\}$ gives $P_{\rm{VI}}$ as the monodromy preserving deformation.
Note that the other three systems have no deformation parameters.

In general, the Fuchsian systems can be classified with the aid of algorithm proposed by Oshima \cite{O}.
The systems with four accessary parameters are reduced as follows:
\[\begin{array}{l|lll}
	\text{5 singularities}& 11,11,11,11,11\\\hline
	\text{4 singularities}& 21,21,111,111& 31,22,22,1111& 22,22,22,211\\\hline
	\text{3 singularities}& 211,1111,1111& 221,221,11111& 32,11111,111111\\
	& 222,222,2211& 33,2211,111111& 44,2222,22211\\
	& 44,332,11111111& 55,3331,22222& 66,444,2222211
\end{array}\]
The system with $\{11,11,11,11,11\}$ corresponds to the Garnier system in two variables \cite{G}.
And the systems with four singularities correspond to four-dimensional Painlev\'{e} equations, which are investigated by Sakai \cite{Sak}.
Among them, the system with $\{31,22,22,1111\}$ corresponds to the system \eqref{Eq:CP6} of the case $n=2$.
In this article, we consider its natural extension.
Namely, we consider the Fuchsian system with the spectral type $\{(n,n),(n,n),(2n-1,1),(1^{2n})\}$ and show that its monodromy preserving deformation gives the system \eqref{Eq:CP6}.

\begin{rem}
The choice of a spectral type $\{(n,n),(n,n),(2n-1,1),(1^{2n})\}$ is suggested by the recent work of Oshima {\rm\cite{O}}.
According to it, a Fuchsian system with this spectral type corresponds to a Kac-Moody root system with the following Dynkin diagram{\rm:}
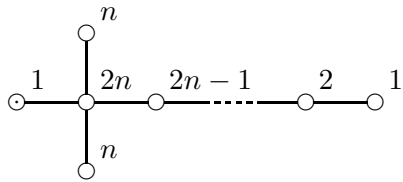
\[
%\fbox{
	\begin{picture}(140,65)
		\put(0,30){\circle{6}}\put(0,30){\circle*{1}}\put(5,35){\small$1$}
		\put(3,30){\line(1,0){20}}
		\put(26,4){\circle{6}}\put(31,9){\small$n$}
		\put(26,7){\line(0,1){20}}
		\put(26,56){\circle{6}}\put(31,61){\small$n$}
		\put(26,33){\line(0,1){20}}
		\put(26,30){\circle{6}}\put(31,35){\small$2n$}
		\put(29,30){\line(1,0){20}}
		\put(52,30){\circle{6}}\put(57,35){\small$2n-1$}
		\put(55,30){\line(1,0){15}}
		\multiput(70,30)(4,0){5}{\line(1,0){2}}
		\put(90,30){\line(1,0){15}}
		\put(108,30){\circle{6}}\put(113,35){\small$2$}
		\put(111,30){\line(1,0){20}}
		\put(134,30){\circle{6}}\put(139,35){\small$1$}
	\end{picture}
%}
\]
A dotted circle represents a simple root which is not orthogonal to the othet roots.
\end{rem}

\begin{rem}
The Fuchsian system with the spectral type $\{21,21,111,111\}$ corresponds to the fourth order Painlev\'{e} system given in {\rm\cite{FS2}}.
Furthermore the system with the spectral type $\{(n,1),(n,1),(1^{n+1}),(1^{n+1})\}$ is systematically investigated by Tsuda.
It corresponds to the Schlesinger system $\mathcal{H}_{n+1,1}$ given in {\rm\cite{T}}, or equivalently, the higher order Painlev\'{e} system given in {\rm\cite{Su}}.
\end{rem}

The other aim of this article is to investigate a relationship between two origins of the Sasano system, the Lax pair associated with $\mathfrak{so}_{4n+4}[z,z^{-1}]$ and the Fuchsian system with the spectral type $\{(n,n),(n,n),(2n-1,1),(1^{2n})\}$.
It is suggested that those two linear systems are related via a Laplace transformation.
In this article, we show it for the case $n=2$.

This article is organized as follows.
In Section \ref{Sec:Sch}, we introduce a Fuchsian system with the spectral type $\{(n,n),(n,n),(2n-1,1),(1^{2n})\}$ and its monodromy preserving deformation.
In Section \ref{Sec:Main_Thm}, the system \eqref{Eq:CP6} is derived from the Schlesinger system given in Section \ref{Sec:Sch}.
In Section \ref{Sec:Laplace}, we clarify a relation between two linear systems for the fourth order Sasano system with the aid of a Laplace transformation.

% Section 2
\section{Schlesinger system}\label{Sec:Sch}

In this section, following \cite{JMU,Sak}, we introduce a Fuchsian system with the spectral type $\{(n,n),(n,n),(2n-1,1),(1^{2n})\}$ and its monodromy preserving deformation.

Consider a system of linear differential equations on $\mathbb{P}^1(\mathbb{C})$
\begin{equation}\label{Eq:Sch_Lax}
	\frac{\partial}{\partial x}Y(x) = \left(\frac{A_t}{x-t}+\frac{A_1}{x-1}+\frac{A_0}{x}\right)Y(x),\quad
	A_t,A_1,A_0 \in \mathrm{Mat}(2n;\mathbb{C}),
\end{equation}
with regular singularities $x=t,1,0,\infty$.
Here we assume
\begin{enumerate}
\item
The data of eigenvalues of residue matrices is given by
\[\begin{array}{lllllllll}
	\theta_t,& \theta_t,& \ldots,& \theta_t,& 0,& \ldots,& 0& \text{at}& x=t,\\
	\theta_1,& \theta_1,& \ldots,& \theta_1,& 0,& \ldots,& 0& \text{at}& x=1,\\
	\theta_0,& 0,& \ldots,& 0,& 0,& \ldots,& 0& \text{at}& x=0,\\
	\kappa_1,& \kappa_2,& \ldots,& \kappa_n,& \kappa_{n+1},& \ldots,& \kappa_{2n}& \text{at}& x=\infty.
\end{array}\]
\item
Each residue matrix can be diagonalized.
\end{enumerate}
Note that the Fuchsian relation $n\theta_t+n\theta_1+\theta_0+\sum_{i=1}^{2n}\kappa_i=0$ is satisfied.
The monodromy preserving deformation of the system \eqref{Eq:Sch_Lax} is described as the Schlesinger system
\begin{equation}\label{Eq:Sch}
	\frac{\partial A_t}{\partial t} = -\frac{[A_t,A_0]}{t} - \frac{[A_t,A_1]}{t-1},\quad
	\frac{\partial A_1}{\partial t} = \frac{[A_t,A_1]}{t-1},\quad
	\frac{\partial A_0}{\partial t} = \frac{[A_t,A_0]}{t}.
\end{equation}
Note that the residue matrix $A_{\infty}=-A_t-A_1-A_0$ at $x=\infty$ is a constant matrix.
The system \eqref{Eq:Sch} can be expressed as a Hamiltonian system
\begin{equation}\label{Eq:Sch_Ham}
	\frac{\partial A_{\xi}}{\partial t} = \{K,A_{\xi}\}\quad (\xi=t,1,0),\quad
	K = \frac{\mathrm{tr}A_tA_1}{t-1} + \frac{\mathrm{tr}A_tA_0}{t},
\end{equation}
with the Poisson bracket
\[
	\{(A_{\xi})_{k,l},(A_{\xi'})_{r,s}\} = \delta_{\xi,\xi'}\{\delta_{r,l}(A_{\xi})_{k,s}-\delta_{k,s}(A_{\xi})_{r,l}\},
\]
where $\delta_{i,j}$ stands for the Kronecker delta.

We consider a gauge transformation $\widetilde{A}_{\xi}=G^{-1}A_{\xi}G$ $(\xi=t,1,0,\infty)$ such that
\[
	\widetilde{A}_0 = \begin{bmatrix}\theta_0&a^{(0)}_2&\ldots&a^{(0)}_{2n}\\0&0&\ldots&0\\\vdots&\vdots&\ddots&\vdots\\0&0&\ldots&0\end{bmatrix},\quad
	\widetilde{A}_{\infty} = \begin{bmatrix}\kappa_1&&&\bm{O}\\a^{(\infty)}_2&\kappa_2&&\\\vdots&&\ddots&\\a^{(\infty)}_{2n}&\bm{O}&&\kappa_{2n}\end{bmatrix}.
\]
Here the matrix $G$ is decomposed into a product of two matrices as $G=G_1G_2$, where $G_1^{-1}A_{\infty}G_1$ is a diagonal matrix and $G_2$ is a lower triangle matrix of which all entries on the diagonals are one.
Then the system \eqref{Eq:Sch_Ham} is transformed into
\begin{equation}\label{Eq:Sch_Ham_Sak}
	\frac{\partial\widetilde{A}_{\xi}}{\partial t} = \{K,\widetilde{A}_{\xi}\}\quad (\xi=t,1,0),\quad
	K = \frac{\mathrm{tr}\widetilde{A}_t\widetilde{A}_1}{t-1} + \frac{\mathrm{tr}\widetilde{A}_t\widetilde{A}_0}{t},
\end{equation}
with the Poisson bracket
\begin{equation}\label{Eq:Sch_Ham_Sak_Poi}
	\{(\widetilde{A}_{\xi})_{k,l},(\widetilde{A}_{\xi'})_{r,s}\} = \delta_{\xi,\xi'}\{\delta_{r,l}(\widetilde{A}_{\xi})_{k,s}-\delta_{k,s}(\widetilde{A}_{\xi})_{r,l}\}.
\end{equation}
Note that the following relation is satisfied:
\begin{equation}\label{Eq:Res_Mat_Sak}
	\widetilde{A}_t + \widetilde{A}_1 + \widetilde{A}_0 + \widetilde{A}_{\infty} = 0.
\end{equation}

In order to derive the canonical Hamiltonian system from \eqref{Eq:Sch_Ham_Sak}, we use the method established in \cite{JMMS}.
Consider a decomposition of matrices $A_{\xi}$ $(\xi=t,1)$ as
\[
	\widetilde{A}_{\xi} = \begin{bmatrix}I_n\\B_{\xi}\end{bmatrix}\begin{bmatrix}\theta_{\xi}I_n-C_{\xi}B_{\xi},&C_{\xi}\end{bmatrix},
\]
where $B_{\xi}=[b^{(\xi)}_{i,j}]$ and $C_{\xi}=[c^{(\xi)}_{i,j}]$ are $n\times n$ matrices.
Then we can regard $b^{(\xi)}_{i,j}$ and $c^{(\xi)}_{j,i}$ as canonical variables.
In fact, the Poisson bracket
\begin{equation}\label{Eq:Sch_Ham_Sak_Poi_JMMS}
	\{b^{(\xi)}_{i,j},c^{(\xi)}_{j,i}\} = -1\quad (i,j=1,\ldots,n; \xi=t,1),\quad \{\text{otherwise}\} = 0.
\end{equation}
implies the one \eqref{Eq:Sch_Ham_Sak_Poi}.

The number of accessary parameters of the system \eqref{Eq:Sch_Lax} is equal to $2n$.
Therefore the system \eqref{Eq:Sch_Ham_Sak} with \eqref{Eq:Res_Mat_Sak} can be rewritten into the Hamiltonian system of order $2n$, which is just equivalent to \eqref{Eq:CP6} as we prove below.

% Section 3
\section{Sasano system}\label{Sec:Main_Thm}

Under the system \eqref{Eq:CP6}, we define independent and dependent variables by
\begin{equation}\label{Eq:DepVarTransf_DStoMPD}
	t = 1 - \frac{1}{s},\quad \lambda_i = 1 - \frac{q_i}{s},\quad \mu_i = -sp_i\quad (i=1,\ldots,n).
\end{equation}
Then they satisfy a Hamiltonian system
\begin{equation}\begin{split}\label{Eq:CP6_MPD}
	&t(t-1)\frac{d\lambda_i}{dt} = \frac{\partial H}{\partial\mu_i},\quad
	t(t-1)\frac{d\mu_i}{dt} = -\frac{\partial H}{\partial\lambda_i}\quad (i=1,\ldots,n),\\
	&H = \sum_{i=1}^{n}H_i[\lambda_i,\mu_i;\kappa_{i,1},\kappa_{i,0},\kappa_{i,s},\kappa_{i,\infty};t]\\
	&\qquad + \sum_{1\leq i<j\leq n}2\lambda_i\mu_i(\lambda_j-1)\{(\lambda_j-t)\mu_j+\alpha_{2j}\},
\end{split}\end{equation}
where $\kappa_{i,s},\kappa_{i,1},\kappa_{i,0},\kappa_{i,\infty}$ are the parameters defined by \eqref{Eq:CP6_Param}.
In this section, we derive the system \eqref{Eq:CP6_MPD} from the one \eqref{Eq:Sch_Ham_Sak} with \eqref{Eq:Res_Mat_Sak}.

Let $\Delta^{i_1,\ldots,i_r}_{j_1,\ldots,j_r}(A)$ be a minor determinant of $A$ for $(i_1,\ldots,i_r)$-th row and $(j_1,\ldots,j_r)$-th column.
Then we arrive at
\begin{thm}\label{Main_Thm}
Under the system \eqref{Eq:Sch_Ham_Sak} with \eqref{Eq:Res_Mat_Sak} and \eqref{Eq:Sch_Ham_Sak_Poi_JMMS}, we set
\begin{equation}\begin{split}\label{Cano_Coor}
	\mu_i &= (-1)^{n-i}t^{-1}\frac{\Delta^{1,i+1,\ldots,n}_{i,i+1,\ldots,n}(C_1)}{\Delta^{i+1,\ldots,n}_{i+1,\ldots,n}(C_t)}\sum_{k=1}^{i}\frac{\Delta^{i,i+1,\ldots,n}_{k,i+1,\ldots,n}(C_t)}{\Delta^{i,\ldots,n}_{i,\ldots,n}(C_t)}b^{(t)}_{k,1},\\
	\lambda_i &= (-1)^{n-i+1}t\frac{\Delta^{1,i+1,\ldots,n}_{i,i+1,\ldots,n}(C_t)}{\Delta^{1,i+1,\ldots,n}_{i,i+1,\ldots,n}(C_1)}\quad (i=1,\ldots,n).
\end{split}\end{equation}
Then those variables are found out to be canonical coordinates of a $2n$-dimensional system with the Poisson bracket
\[
	\{\mu_i,\lambda_j\} = \delta_{i,j},\quad
	\{\mu_i,\mu_j\} = \{\lambda_i,\lambda_j\} = 0\quad (i,j=1,\ldots,n).
\]
Furthermore they satisfy the system \eqref{Eq:CP6_MPD} with the parameters
\begin{equation}\begin{split}\label{Parameter}
	&\alpha_1 = -\theta_t,\quad
	\alpha_2 = -\kappa_{n+1},\quad
	\alpha_{2i-1} = \theta_t + \theta_1 + \kappa_i + \kappa_{n+i-1},\\
	&\alpha_{2i} = -\theta_t - \theta_1 - \kappa_i - \kappa_{n+i}\quad (i=2,\ldots,n),\\
	&\alpha_{2n+1} = \theta_t + \theta_1 + \kappa_1 + \kappa_{2n},\quad
	\alpha_{2n+2} = -\kappa_1 + \kappa_{2n} + 1.
\end{split}\end{equation}
\end{thm}

% Section 3.1
\subsection{Canonical coordinates}

In this subsection, we prove the first half of Theorem \ref{Main_Thm}.

We can show $\{\mu_i,\lambda_j\}=\delta_{i,j}$ as follows.
Denoting $\Delta^{j+1,\ldots,n}_{j,\ldots,l-1,l+1,\ldots,n}$ by $\Delta^{j+1,\ldots,n}_{j,\ldots,\widehat{l},\ldots,n}$, we have
\begin{equation}\label{Eq:Prf_Cano_Coor}
	\{\mu_i,\lambda_j\} = \frac{\Delta^{1,i+1,\ldots,n}_{i,i+1,\ldots,n}(C_1)}{\Delta^{i+1,\ldots,n}_{i+1,\ldots,n}(C_t)}\sum_{k=1}^{i}\frac{\Delta^{i,i+1,\ldots,n}_{k,i+1,\ldots,n}(C_t)}{\Delta^{i,\ldots,n}_{i,\ldots,n}(C_t)}\sum_{l=j}^{n}(-1)^{l-j}\frac{\Delta^{j+1,\ldots,n}_{j,\ldots,\widehat{l},\ldots,n}(C_t)}{\Delta^{1,j+1,\ldots,n}_{j,j+1,\ldots,n}(C_1)}\delta_{k,l}.
\end{equation}
If $i<j$, the right-hand side of \eqref{Eq:Prf_Cano_Coor} turns to be zero.
If $i=j$, the right-hand side of \eqref{Eq:Prf_Cano_Coor} turns to be one.
If $j<i$, then we have
\[
	(\text{RHS of \eqref{Eq:Prf_Cano_Coor}}) = \frac{\Delta^{1,i+1,\ldots,n}_{i,i+1,\ldots,n}(C_1)}{\Delta^{i+1,\ldots,n}_{i+1,\ldots,n}(C_t)}\sum_{k=j}^{i}(-1)^{k-j}\frac{\Delta^{i,i+1,\ldots,n}_{k,i+1,\ldots,n}(C_t)\Delta^{j+1,\ldots,n}_{j,\ldots,\widehat{k},\ldots,n}(C_t)}{\Delta^{i,\ldots,n}_{i,\ldots,n}(C_t)\Delta^{1,j+1,\ldots,n}_{j,j+1,\ldots,n}(C_1)}.
\]
On the other hand, we obtain
\[\begin{split}
	&\sum_{k=j}^{i}(-1)^{k-j}\Delta^{i,i+1,\ldots,n}_{k,i+1,\ldots,n}(C_t)\Delta^{j+1,\ldots,n}_{j,\ldots,\widehat{k},\ldots,n}(C_t)\\
	&\quad = \begin{vmatrix}\Delta^{i,i+1,\ldots,n}_{j,i+1,\ldots,n}(C_t)&\Delta^{i,i+1,\ldots,n}_{j+1,i+1,\ldots,n}(C_t)&\ldots&\Delta^{i,\ldots,n}_{i,\ldots,n}(C_t)&0&\ldots&0\\c^{(t)}_{j+1,j}&c^{(t)}_{j+1,j+1}&\ldots&c^{(t)}_{j+1,i}&c^{(t)}_{j+1,i+1}&\ldots&c^{(t)}_{j+1,n}\\\vdots&\vdots&\ddots&\vdots&\vdots&\ddots&\vdots\\c^{(t)}_{n,j}&c^{(t)}_{n,j+1}&\ldots&c^{(t)}_{n,i}&c^{(t)}_{n,i+1}&\ldots&c^{(t)}_{n,n}\end{vmatrix}\\
	&\quad = \sum_{l=i}^{n}(-1)^{l-i}\Delta^{i,\ldots,\widehat{l},\ldots,n}_{i+1,\ldots,n}(C_t)\Delta^{l,j+1,\ldots,n}_{j,j+1,\ldots,n}(C_t)\\
	&\quad = 0.
\end{split}\]
Hence the right-hand side of \eqref{Eq:Prf_Cano_Coor} turns to be zero.

We can show $\{\mu_i,\mu_j\}=0$ and $\{\lambda_i,\lambda_j\}=0$ immediately because rational expressions $\mu_i$ and $\lambda_i$ defined by \eqref{Cano_Coor} do not contain the canonical variables $c^{(t)}_{1,k}$ and $b^{(t)}_{k,1}$, respectively.

% Section 3.2
\subsection{Derivation of the Sasano system}

In this subsection, we prove the second half of Theorem \ref{Main_Thm}.

Under the system \eqref{Eq:Sch_Ham_Sak} with \eqref{Eq:Res_Mat_Sak}, the dependent variables $\mu_i,\lambda_i$ given by \eqref{Cano_Coor} satisfy
\[\begin{split}
	&\frac{\partial \mu_i}{\partial t} = \{\widetilde{K},\mu_i\},\quad
	\frac{\partial \lambda_i}{\partial t} = \{\widetilde{K},\lambda_i\}\quad (i=1,\ldots,n),\\
	&\widetilde{K} = \frac{\mathrm{tr}\widetilde{A}_t\widetilde{A}_1}{t-1} + \frac{\mathrm{tr}\widetilde{A}_t\widetilde{A}_0}{t} + \sum_{i=1}^{n}\frac{\mu_i\lambda_i}{t}.
\end{split}\]
Hence it is enough to verify that the Hamiltonian $\widetilde{K}$ is transformed into the one $H$ given by \eqref{Eq:CP6_MPD} via the transformation \eqref{Cano_Coor} and \eqref{Parameter}.

First we consider a partition of residue matrix
\[
	\widetilde{A}_{\xi} = \begin{bmatrix}A^{(\xi)}_{11}&A^{(\xi)}_{12}&A^{(\xi)}_{13}\\A^{(\xi)}_{21}&A^{(\xi)}_{22}&A^{(\xi)}_{23}\\A^{(\xi)}_{31}&A^{(\xi)}_{32}&A^{(\xi)}_{33}\end{bmatrix}\quad (\xi=t,1,0,\infty),
\]
where each block $A^{(\xi)}_{ij}$ is an $n_i\times n_j$ matrix with $(n_1,n_2,n_3)=(1,n-1,n)$.
With this block form, the relation \eqref{Eq:Res_Mat_Sak} is described as
\begin{equation}\label{Eq:Res_Mat_block}
	A^{(t)}_{ij} + A^{(1)}_{ij} + A^{(0)}_{ij} + A^{(\infty)}_{ij}\quad (i,j=1,2,3).
\end{equation}
The Hamiltonian $\widetilde{K}$ is given by
\begin{equation}\begin{split}\label{Eq:Sch_Ham_block}
	\mathrm{tr}\widetilde{A}_t\widetilde{A}_1 &= \sum_{i=1}^{3}\sum_{j=1}^{3}\mathrm{tr}A^{(t)}_{ij}A^{(1)}_{ji},\\
	\mathrm{tr}\widetilde{A}_t\widetilde{A}_0 &= \theta_0A^{(t)}_{11} - \mathrm{tr}A^{(t)}_{21}(A^{(t)}_{12}+A^{(1)}_{12}) - \mathrm{tr}A^{(t)}_{31}(A^{(t)}_{13}+A^{(1)}_{13}).
\end{split}\end{equation}
Note that $A^{(0)}_{2j}=A^{(0)}_{3j}=0$ and $A^{(\infty)}_{12}=A^{(\infty)}_{13}=A^{(\infty)}_{23}=A^{(\infty)}_{32}=0$.

Next we rewrite the Hamiltonian given by \eqref{Eq:Sch_Ham_block} into the one expressed in terms of the matrices $B_t,C_t,B_1,C_1$.
Let $E_1=\mathrm{diag}[1,0,\ldots,0]$ and $E_{2n}=\mathrm{diag}[0,1,\ldots,1]$.
Then the relation \eqref{Eq:Res_Mat_block} implies
\[
	E_1(C_tB_t+C_1B_1)E_1 - (\theta_t+\theta_1+\theta_0+\kappa_1)E_1 = 0,
\]
for $(i,j)=(1,1)$;
\[
	E_{2n}(C_tB_t+C_1B_1)E_{2n} - \mathrm{diag}[0,\theta_t+\theta_1+\kappa_2,\ldots,\theta_t+\theta_1+\kappa_n] = 0,
\]
for $(i,j)=(3,3)$;
\[
	E_{2n}(C_t+C_1) = 0,
\]
for $(i,j)=(2,3)$;
\[
	B_tC_t + B_1C_1 + \mathrm{diag}[\kappa_{n+1},\ldots,\kappa_{2n}] = 0,
\]
for $(i,j)=(3,3)$.
We obtain from them
\[\begin{split}
	\mathrm{tr}\widetilde{A}_t\widetilde{A}_1 &= (\mathrm{tr}E_1C_tB_t)(\mathrm{tr}E_1C_tB_1) - (\mathrm{tr}E_1C_tB_t)(\mathrm{tr}E_1C_1B_t) - 2(\mathrm{tr}E_1C_tB_t)^2\\
	&\quad - \mathrm{tr}E_1C_1(B_t-B_1)E_{2n}C_tB_t - \mathrm{tr}E_1C_t(B_t-B_1)E_{2n}C_tB_1\\
	&\quad + (3\theta_t+\theta_1+2\theta_0+2\kappa_1)\mathrm{tr}E_1C_tB_t - (\theta_t+\theta_0+\kappa_1)\mathrm{tr}E_1C_tB_1\\
	&\quad + \theta_t\mathrm{tr}E_1C_1B_t + n\theta_t\theta_1 - \frac{1}{2}\sum_{i=2}^{n}(\theta_t+\theta_1-\kappa_i)(\theta_t+\theta_1+\kappa_i)\\
	&\quad + \frac{1}{2}\sum_{i=1}^{n}\kappa_{n+i}^2 - \frac{1}{2}(\theta_t+\theta_1+\theta_0+\kappa_1)^2 - \theta^t(\theta_t+\theta_1+\theta_0+\kappa_1),\\
\end{split}\]
and
\[\begin{split}
	\mathrm{tr}\widetilde{A}_t\widetilde{A}_0 &= (\mathrm{tr}E_1C_tB_t)(\mathrm{tr}E_1C_1B_t) + (\mathrm{tr}E_1C_tB_t)^2 + \mathrm{tr}E_1C_1(B_t-B_1)E_{2n}C_tB_t\\
	&\quad - (\theta_t+\theta_0)\mathrm{tr}E_1C_tB_t - \theta_t\mathrm{tr}E_1C_1B_t + \theta_t\theta_0.
\end{split}\]

In order to derive the Hamiltonian $H$ given by \eqref{Eq:CP6_MPD}, we introduce the following lemma.
\begin{lem}
We have relations
\[\begin{split}
	&\mathrm{tr}E_1C_tB_t = -\sum_{i=1}^{n}\lambda_i\mu_i,\quad
	\mathrm{tr}E_1C_1B_t = t\sum_{i=1}^{n}\mu_i,\\
	&\mathrm{tr}E_1C_tB_1 = -\frac{1}{t}\sum_{i=1}^{n}\lambda_i(\lambda_i\mu_i+\beta_i),\\
	&\mathrm{tr}E_1C_1(B_t-B_1)E_{2n}C_tB_t\\
	&\quad = t\sum_{i=1}^{n}\mu_i\left\{-\sum_{j=1}^{i-1}(\lambda_j\mu_j+\beta_j)-\beta_i-\kappa_{n+i}+\sum_{j=i+1}^{n}\lambda_j\mu_j\right\},\\
	&\mathrm{tr}E_1C_t(B_t-B_1)E_{2n}C_tB_1\\
	&\quad = -\frac{1}{t}\sum_{i=1}^{n}\lambda_i(\lambda_i\mu_i+\beta_i)\left\{\sum_{j=1}^{i-1}\lambda_j\mu_j-\beta_i-\kappa_{n+i}-\sum_{j=i+1}^{n}(\lambda_j\mu_j+\beta_j)\right\},
\end{split}\]
where
\[
	\beta_1 = -\kappa_{n+1},\quad
	\beta_i = -\theta_t - \theta_1 - \kappa_i - \kappa_{n+i}\quad (i=2,\ldots,n).
\]
\end{lem}

\begin{proof}
We only prove the first relation here.
The other ones can be proved in a similar way.

We take an $n\times n$ matrix
\[
	P = \begin{bmatrix}f_{1,1}&&&O\\f_{2,1}&f_{2,2}&&\\\vdots&\vdots&\ddots&\\f_{n,1}&f_{n,2}&\ldots&f_{n,n}\end{bmatrix},\quad
	f_{i,k} = (-1)^{n-i}t^{-1}\frac{\Delta^{1,i+1,\ldots,n}_{i,i+1,\ldots,n}(C_1)\Delta^{i,i+1,\ldots,n}_{k,i+1,\ldots,n}(C_t)}{\Delta^{i+1,\ldots,n}_{i+1,\ldots,n}(C_t)\Delta^{i,\ldots,n}_{i,\ldots,n}(C_t)}.
\]
Its inverse matrix of is given by
\[
	P^{-1} = \begin{bmatrix}g_{1,1}&&&O\\g_{2,1}&g_{2,2}&&\\\vdots&\vdots&\ddots&\\g_{n,1}&g_{n,2}&\ldots&g_{n,n}\end{bmatrix},\quad
	g_{k,i} = (-1)^{n-k}t\frac{\Delta^{i+1,\ldots,n}_{i,\ldots,\widehat{k},\ldots,n}(C_t)}{\Delta^{1,i+1,\ldots,n}_{i,i+1,\ldots,n}(C_1)}.
\]
Note that we derive an explicit formula of $P^{-1}$ by using the Pl\"{u}cker relations for matrices.
Then we can rewrite \eqref{Cano_Coor} into
\begin{equation}\label{Eq:Prf_Lem_Main_Thm_1}
	\begin{bmatrix}\mu_1\\\vdots\\\mu_n\end{bmatrix} = P\begin{bmatrix}b^{(t)}_{1,1}\\\vdots\\b^{(t)}_{n,1}\end{bmatrix},\quad
	\begin{bmatrix}\lambda_1&\ldots&\lambda_n\end{bmatrix} = -\begin{bmatrix}c^{(t)}_{1,1}&\ldots&c^{(t)}_{n,1}\end{bmatrix}P^{-1}.
\end{equation}
On the other hand, an adjoint action of $P$ implies
\begin{equation}\label{Eq:Prf_Lem_Main_Thm_2}
	\mathrm{tr}E_1C_tB_t = \mathrm{tr}P\begin{bmatrix}b^{(t)}_{1,1}\\\vdots\\b^{(t)}_{n,1}\end{bmatrix}\begin{bmatrix}c^{(t)}_{1,1}&\ldots&c^{(t)}_{n,1}\end{bmatrix}P^{-1}.
\end{equation}
Combining \eqref{Eq:Prf_Lem_Main_Thm_1} and \eqref{Eq:Prf_Lem_Main_Thm_2}, we derive the first relation.
\end{proof}

% Section 4
\section{Laplace transformation}\label{Sec:Laplace}

As is seen in the previous section, the system \eqref{Eq:CP6} is derived from the Fuchsian system.
On the other hand, in the previous work \cite{FS1}, it was also derived from the Lax pair associated with the loop algebra $\mathfrak{so}_{4n+4}[z,z^{-1}]$.
In this section, we clarify a relation between those two linear systems with the aid of a Laplace transformation for the case $n=2$.

We recall the definition of the loop algebra $\mathfrak{so}_{2N}$ for $N\geq3$.
Let $E_{i,j}$ be a $2N\times2N$ matrix with 1 on the $(i,j)$-th entry and zeros elsewhere.
We also set $X_{i,j}=E_{i,j}-E_{2N+1-j,2N+1-i}$.
Then the loop algebra $\mathfrak{so}_{2N}[z,z^{-1}]$ is generated by
\[\begin{split}
	&e_0=zX_{2N-1,1},\quad e_i=X_{i,i+1}\quad (i=1,\ldots,N-1),\quad e_N=X_{N-1,N+1},\\
	&f_0=\frac{1}{z}X_{1,2N-1},\quad f_i=X_{i+1,i}\quad (i=1,\ldots,N-1),\quad f_N=X_{N+1,N-1},
\end{split}\]
and $h_i = X_{i,i}$ $(i=1,\ldots,N)$.
In the following, we use a notation
\[
	e_{i_1,\ldots,i_{n-1},i_n} = \mathrm{ad}e_{i_1}\ldots\mathrm{ad}e_{i_{n-1}}(e_{i_n}),\quad
	\mathrm{ad}e_i(e_j) = [e_i,e_j].
\]
Note that the algebra $\mathfrak{so}_{2N}$ is defined by
\[
	\mathfrak{so}_{2N} = \left\{X\in\mathrm{Mat}(2n;\mathbb{C})\bigm|JX+{}^tXJ=0\right\},\quad
	J = \sum_{i=1}^{2N}E_{i,2N+1-i}.
\]

The system \eqref{Eq:CP6} of the case $n=2$ is given as the compatibility condition of the Lax pair
\begin{equation}\label{Eq:Lax_so12}
	z\frac{\partial}{\partial z}\Psi_{12}(z) = M_{12}(z)\Psi_{12}(z),\quad
	\frac{\partial}{\partial s}\Psi_{12}(z) = B_{12}(z)\Psi_{12}(z).
\end{equation}
The matrix $M_{12}(z)\in\mathfrak{so}_{12}[z,z^{-1}]$ is described as
\[
	M_{12}(z) = \sum_{i=1}^{6}\varepsilon_ih_i - e_0 + \sum_{i=1}^{6}\varphi_ie_i - e_{1,2} - e_{2,3} - e_{3,4} - e_{4,5} - e_{4,6},
\]
where
\[\begin{split}
	&\varepsilon_1 = \frac{1}{2}(-1+\alpha_0-\alpha_1),\quad
	\varepsilon_2 = \frac{1}{2}(-1+\alpha_0+\alpha_1),\quad
	\varepsilon_3 = \frac{1}{2}(-1+\alpha_0+\alpha_1+2\alpha_2),\\
	&\varepsilon_4 = \frac{1}{2}(-2\alpha_4-\alpha_5-\alpha_6),\quad
	\varepsilon_5 = \frac{1}{2}(-\alpha_5-\alpha_6),\quad
	\varepsilon_6 = \frac{1}{2}(\alpha_5-\alpha_6),
\end{split}\]
and
\[
	\varphi_1 = s - q_1,\quad
	\varphi_2 = p_1,\quad
	\varphi_3 = q_1 - q_2,\quad
	\varphi_4 = p_2,\quad
	\varphi_5 = q_2 - 1,\quad
	\varphi_6 = q_2.
\]
The matrix $B_{12}(z)\in\mathfrak{so}_{12}[z,z^{-1}]$ is described as
\[\begin{split}
	B_{12}(z) &= \sum_{i=1}^{6}u_ih_i + \sum_{i=0}^{6}v_ie_i + e_{0,2} + v_7e_{2,3} + v_8e_{3,4} + v_9e_{4,5} + v_{10}e_{4,6} + v_{11}e_{2,3,4},
\end{split}\]
where the coefficients $u_i,v_i$ are polynomials in $(q_1,q_2,p_1,p_2)$; we do not give their explicit formulas here.
In this section, we reduce it to a Fuchsian system with a spectral type $\{31,22,22,1111\}$.

% Section 4.1
\subsection{From $\mathfrak{so}_{12}[z,z^{-1}]$ to $\mathfrak{so}_{10}[z,z^{-1}]$}

Under the system \eqref{Eq:Lax_so12}, we consider a gauge transformation
\[
	\Psi_{12}(z) \to \tau_{12}(z)\Psi^*_{12}(z),
\]
where a function $\tau_{12}(z)$ satisfies
\[
	z\frac{\partial}{\partial z}\log\tau_{12}(z) = \varepsilon_1 + 1,\quad
	\frac{\partial}{\partial s}\log\tau_{12}(z) = u_1.
\]
We also consider a Laplace transformation
\[
	\frac{\partial}{\partial z}\Psi^*_{12}(z) \to \zeta\Phi_{12}(\zeta^{-1}),\quad
	z\Psi^*_{12}(z) \to -\frac{\partial}{\partial \zeta}\Phi_{12}(\zeta^{-1}),
\]
and a M\"{o}bius transformation $\zeta\to z^{-1}$.
Then we obtain
\begin{equation}\label{Eq:Lax_so12_afterLaplace}
	z\frac{\partial}{\partial z}\Phi_{12}(z) = N_{12}(z)\Phi_{12}(z),\quad
	\frac{\partial}{\partial s}\Phi_{12}(z) = C_{12}(z)\Phi_{12}(z),
\end{equation}
with
\[\begin{split}
	N_{12}(z) &= (I_{12}-zM_{12,1})^{-1}(M_{12,0}-\varepsilon_1I_{12}),\\
	C_{12}(z) &= B_{12,0} - u_1I_{12} + zB_{12,1}(I_{12}-zM_{12,1})^{-1}(M_{12,0}-\varepsilon_1I_{12}),
\end{split}\]
where $M_{12}(z)=M_{12,0}+zM_{12,1}$ and $B_{12}(z)=B_{12,0}+zB_{12,1}$.
Note that $(M_{12,1})^2=O$, namely, $(I_{12}-zM_{12,1})^{-1}=I_{12}+zM_{12,1}$.
The first columns of $N_{12}(z)$ and $C_{12}(z)$ are both equivalent to the zero vectors.
Hence we can reduce the system \eqref{Eq:Lax_so12_afterLaplace} to the one with $11\times11$ matrices
\[
	z\frac{\partial}{\partial z}\Psi_{11}(z) = M_{11}(z)\Psi_{11}(z),\quad
	\frac{\partial}{\partial s}\Psi_{11}(z) = B_{11}(z)\Psi_{11}(z),
\]
or equivalently
\begin{equation}\label{Eq:Lax_mat11}
	z\frac{\partial}{\partial z}\Psi^{-1}_{11}(z) = -\Psi^{-1}_{11}(z)M_{11}(z),\quad
	\frac{\partial}{\partial s}\Psi^{-1}_{11}(z) = -\Psi^{-1}_{11}(z)B_{11}(z).
\end{equation}

Under the system \eqref{Eq:Lax_mat11}, we consider a gauge transformation
\[
	\Psi^{-1}_{11}(z) \to \tau_{11}(z)\Psi^*_{11}(z),\quad
	z\frac{\partial}{\partial z}\log\tau_{11}(z) = -2\varepsilon_1 - 1,\quad
	\frac{\partial}{\partial s}\log\tau_{11}(z) = -2u_1,
\]
a Laplace transformation
\[
	\frac{\partial}{\partial z}\Psi^*_{11}(z) \to \zeta\Phi_{11}(\zeta^{-1}),\quad
	z\Psi^*_{11}(z) \to -\frac{\partial}{\partial \zeta}\Phi_{11}(\zeta^{-1}),
\]
and a M\"{o}bius transformation $\zeta\to z^{-1}$.
Then we obtain
\begin{equation}\label{Eq:Lax_mat11_afterLaplace}
	z\frac{\partial}{\partial z}\Phi_{11}(z) = -\Phi_{11}(z)N_{11}(z),\quad
	\frac{\partial}{\partial s}\Phi_{11}(z) = -\Phi_{11}(z)C_{11}(z),
\end{equation}
with
\[\begin{split}
	N_{11}(z) &= (M_{11,0}+2\varepsilon_1I_{11})(I_{11}+zM_{11,1})^{-1},\\
	C_{11}(z) &= B_{11,0} + 2u_1I_{11} + z(M_{11,0}+2\varepsilon_1I_{11})(I_{11}+zM_{11,1})^{-1}B_{11,1},
\end{split}\]
where $M_{11}(z)=M_{11,0}+zM_{11,1}$ and $B_{11}(z)=B_{11,0}+zB_{11,1}$.
Note that $(M_{11,1})^2=O$, namely, $(I_{11}+zM_{11,1})^{-1}=I_{11}-zM_{11,1}$.
The 11-th rows of $N_{11}(z)$ and $C_{11}(z)$ are both equivalent to the zero vectors.
Hence we can reduce the system \eqref{Eq:Lax_mat11_afterLaplace} to the one associated with $\mathfrak{so}_{10}[z,z^{-1}]$.

Furthermore, we consider a Dynkin diagram automorphism
\[
	e_i \to e_{5-i},\quad
	h_i \to h_{5-i}\quad (i=0,\ldots,5),\quad
	z\frac{\partial}{\partial z} \to z\frac{\partial}{\partial z} + \frac{1}{2}\sum_{i=1}^{5}h_i.
\]
We finally obtain
\begin{equation}\label{Eq:Lax_so10}
	z\frac{\partial}{\partial z}\Psi_{10}(z) = M_{10}(z)\Psi_{10}(z),\quad
	\frac{\partial}{\partial s}\Psi_{10}(z) = B_{10}(z)\Psi_{10}(z).
\end{equation}
The matrix $M_{10}(z)\in\mathfrak{so}_{10}[z,z^{-1}]$ is described as
\[\begin{split}
	M_{10}(z) &= \sum_{i=1}^{5}(-\varepsilon_{7-i}-\frac{1}{2})h_i + \sum_{i=0}^{4}\varphi_{6-i}e_i + \{(q_i-s)p_1-\alpha_1\}e_5\\
	&\quad + e_{0,2} + e_{1,2} + e_{2,3} + e_{3,4} + (q_1-s)e_{3,5} - e_{5,3,4}.
\end{split}\]
The matrix $B_{10}(z)\in\mathfrak{so}_{10}[z,z^{-1}]$ is described as
\[\begin{split}
	B_{10}(z) &= \sum_{i=1}^{5}u'_ih_i + \sum_{i=0}^{5}v'_ie_i + v'_6e_{0,2} + v'_7e_{1,2} + v'_8e_{2,3} + v'_9e_{3,4}\\
	&\quad + v'_{10}e_{3,5} + v'_{11}e_{2,3,4} + v'_{12}e_{2,3,5} + v'_{13}e_{5,3,4} + v'_{14}e_{5,2,3,4},
\end{split}\]
where the coefficients $u'_i,v'_i$ are polynomials in $(q_1,q_2,p_1,p_2)$; we do not give their explicit formulas here.

% Section 4.2
\subsection{From $\mathfrak{so}_{10}[z,z^{-1}]$ to $\mathfrak{so}_8[z,z^{-1}]$}

Similarly as in the previous section, the system \eqref{Eq:Lax_so10} can be reduced to the one associated with $\mathfrak{so}_8[z,z^{-1}]$.
Furthermore, we consider a gauge transformation
\[
	\Psi_8(z) \to \exp(-\frac{1}{\varphi_5}e_1)\exp(-e_4)\exp(h_1\log\varphi_5)\Psi_8(z),
\]
the B\"{a}cklund transformation for the Sasano system
\[
	p_2 \to p_2 + \frac{\alpha_5}{1-q_2},\quad
	\alpha_4 \to \alpha_4 + \alpha_5,\quad
	\alpha_5 \to -\alpha_5,\quad
	\alpha_0 \to \alpha_0 + \alpha_5,
\]
and a Dynkin diagram automorphism $e_1\leftrightarrow e_4$, $h_1\leftrightarrow h_4$.
We finally obtain
\begin{equation}\label{Eq:Lax_so8}
	z\frac{\partial}{\partial z}\Psi_8(z) = M_8(z)\Psi_8(z),\quad
	\frac{\partial}{\partial s}\Psi_8(z) = B_8(z)\Psi_8(z).
\end{equation}
The matrix $M_8(z)\in\mathfrak{so}_8[z,z^{-1}]$ is described as
\[\begin{split}
	M_8(z) &= \sum_{i=1}^{4}\varepsilon''_{i+1}h_i + \sum_{i=0}^{4}\varphi''_ie_i + q_1e_{0,2} + (s-q_2)e_{1,2} + e_{2,3}\\
	&\quad + (1-q_1)e_{2,4} + se_{0,1,2} + e_{0,2,3} + (1-s)e_{1,2,4} + e_{3,2,4},
\end{split}\]
where
\[\begin{split}
	&\varepsilon''_1 = \frac{1}{2}(\alpha_2+2\alpha_3+3\alpha_4+\alpha_5+2\alpha_6-2),\quad
	\varepsilon''_2 = \frac{1}{2}(-\alpha_2-2\alpha_3-\alpha_4-\alpha_5),\\
	&\varepsilon''_3 = \frac{1}{2}(-\alpha_2-\alpha_4-\alpha_5),\quad
	\varepsilon''_4 = \frac{1}{2}(\alpha_2-\alpha_4-\alpha_5),
\end{split}\]
and
\[\begin{split}
	&\varphi''_0 = q_2p_2 + \alpha_4,\quad
	\varphi''_1 = (q_1-s)p_1 + \alpha_0 + \alpha_2,\\
	&\varphi''_2 = q_1 - q_2,\quad
	\varphi''_3 = p_1,\quad
	\varphi''_4 = (q_2-1)p_2 + \alpha_4.
\end{split}\]
The matrix $B_8(z)\in\mathfrak{so}_8[z,z^{-1}]$ is described as
\[\begin{split}
	B_8(z) &= \sum_{i=1}^{4}u''_ih_i + \sum_{i=0}^{4}v''_ie_i + v''_5e_{0,2} + v''_6e_{1,2} + v''_7e_{2,3}\\
	&\quad + v''_8e_{2,4} + v''_9e_{0,2,1} + v''_{10}e_{0,2,3} + v''_{11}e_{1,2,4} + v''_{12}e_{3,2,4},
\end{split}\]
where the coefficients $u''_i,v''_i$ are polynomials in $(q_1,q_2,p_1,p_2)$; we do not give their explicit formulas here.

% Section 4.3
\subsection{From $\mathfrak{so}_8[z,z^{-1}]$ to $\mathfrak{sl}_4$}

Similarly as in the previous section, the system \eqref{Eq:Lax_so8} can be reduced to a Fuchsian system
\begin{equation}\label{Eq:Lax_so7}
	z\frac{\partial}{\partial z}\Psi_7(z) = M_7(z)\Psi_7(z),\quad
	\frac{\partial}{\partial s}\Psi_7(z) = B_7(z)\Psi_7(z),
\end{equation}
with $7\times7$ matrices $M_7(z)=M_{7,0}+zM_{7,1}$ and $B_7(z)=B_{7,0}+zB_{7,1}$.
Then we have $\det(I_7+zM_{7,1})=(z-s)(z-s+1)$.
It follows that the system \eqref{Eq:Lax_so7} can be reduced to a Fuchsian one with $6\times6$ matrices.
It is described as
\begin{equation}\label{Eq:Lax_so6}
	z\frac{\partial}{\partial z}\Psi_6(z) = M_6(z)\Psi_6(z),\quad
	\frac{\partial}{\partial s}\Psi_6(z) = B_6(z)\Psi_6(z),
\end{equation}
with
\[
	M_6(z) = \frac{M_{6,0}+zM_{6,1}+z^2M_{6,2}}{(z-s)(z-s+1)},\quad
	B_6(z) = \frac{B_{6,0}+zB_{6,1}+z^2B_{6,2}}{(z-s)(z-s+1)},
\]
where $M_{6,i},B_{6,i}\in\mathfrak{so}_6$ $(i=0,1,2)$; we do not give their explicit formulas here.

Recall that the algebra $\mathfrak{so}_6$ is isomorphic to the one $\mathfrak{sl}_4$.
With the aid of this fact, we can reduce the system \eqref{Eq:Lax_so6} to a Fuchsian system of fourth order.
Furthermore, we consider a transformation of independent and dependent variables
\[
	x = \frac{z}{s},\quad t = 1 - \frac{1}{s},\quad \lambda_i = 1 - \frac{q_i}{s},\quad \mu_i = -sp_i\quad (i=1,2).
\]
Note that it arises from the transformation \eqref{Eq:DepVarTransf_DStoMPD} of the case $n=2$.
We finally obtain
\begin{equation}\label{Eq:Lax_sl4}
	\frac{\partial}{\partial x}\Psi_4(x) = M_4(x)\Psi_4(x),\quad
	\frac{\partial}{\partial t}\Psi_4(x) = B_4(x)\Psi_4(x),
\end{equation}
with
\[
	M_4(x) = \frac{M_{4,t}}{x-t} + \frac{M_{4,1}}{x-1} + \frac{M_{4,0}}{x},\quad
	B_4(x) = -\frac{M_{4,t}}{x-t} + B_{4,\infty}.
\]
Its compatibility condition implies the system \eqref{Eq:CP6_MPD}.
We do not give explicit formulas of residue matrices here.

By a direct computation, we arrive at
\begin{thm}
The system \eqref{Eq:Lax_sl4} is a Fuchsian system with a spectral type $\{31,22,22,1111\}$.
\end{thm}

\begin{rem}
The system \eqref{Eq:Lax_sl4} can be transformed into the one \eqref{Eq:Sch_Lax} of the case $n=2$ via certain gauge transformation and B\"{a}cklund transformation.
\end{rem}

% Acknowledgement
\section*{Acknowledgement}
The authers are grateful to Professors Masatoshi Noumi, Hidetaka Sakai and Shintarou Yanagida for valuable discussions and advices.

% References

\end{document}